%% file: grp.tex
\documentclass[a4paper]{cedram-smai-jcm}

\usepackage{tikz}
\usepackage{subcaption}

\usepackage{bm} 
\usepackage{pgfplots}
\usepackage{mathtools}
\mathtoolsset{showonlyrefs}

\newcommand{\R}{\mathbb{R}}
\newcommand{\Om}{\Omega}
\newcommand{\Omh}{\Omega_h}
\newcommand{\Ga}{\Gamma}
\newcommand{\Gah}{\Gamma_h}
\newcommand{\tr}{\gamma}
\newcommand{\trh}{\gamma_h}
\newcommand{\bfb}{\mathbf{b}}
\newcommand{\bfd}{\mathbf{d}}
\newcommand{\bfe}{\mathbf{e}}
\newcommand{\bff}{\mathbf{f}}
\newcommand{\bfg}{\mathbf{g}}
\newcommand{\bfu}{\mathbf{u}}
\newcommand{\bfw}{\mathbf{w}}
\newcommand{\bfx}{\mathbf{x}}
\newcommand{\bfz}{\mathbf{z}}
\newcommand{\bfA}{\mathbf{A}}
\newcommand{\bfK}{\mathbf{K}}
\newcommand{\bfM}{\mathbf{M}}
\newcommand{\bftr}{\bm{\gamma}}
\newcommand{\T}{^\mathrm{T}}
\newcommand{\hatx}{\widehat{x}}
\newcommand{\hatT}{\widehat{T}}
\newcommand{\hatsig}{\widehat{\sigma}}
\newcommand{\calTh}{\mathcal{T}_h}
\newcommand{\dx}{\mathrm{d}x}
\newcommand{\dS}{\mathrm{d}\bm{\sigma}}
\newcommand{\dfdx}[2]{\frac{\partial #1}{\partial #2}}
\newcommand{\nb}{\nabla}
\newcommand{\nbg}{\nabla_\Gamma}
\newcommand{\lbg}{\Delta_\Gamma}
\newcommand{\C}{\kappa}
\newcommand{\uDelt}{\mathcal{U}_\delta}

\title[ISOFEM Analysis of a generalized Robin BVP]{Isoparametric finite element analysis of a generalized \\ Robin boundary value problem on curved domains}

\author[D. Edelmann]{\firstname{Dominik} \lastname{Edelmann}}
\address{Mathematisches Institut, Universit{\"a}t T{\"u}bingen, Germany}
\email{edelmann@na.uni-tuebingen.de}

  
\begin{document}

\begin{abstract}
	We study the discretization of an elliptic partial differential equation, posed on a two- or three-dimensional domain with smooth boundary, endowed with a generalized Robin boundary condition which involves the Lap\-lace--Beltrami operator on the boundary surface. The boundary is approximated with piecewise polynomial faces and we use isoparametric finite elements of arbitrary order for the discretization. We derive optimal-order error bounds for this non-conforming finite element method in both $L^2$- and $H^1$-norm. Numerical examples illustrate the theoretical results.
\end{abstract}

\maketitle


\section{Introduction}
\subsection{The generalized Robin boundary value problem}
In this paper, we study the following second-order partial differential equation endowed with a boundary condition including the Laplace--Beltrami operator
\begin{align}\label{eq: GRP}
\left\{ \begin{aligned}
	- \Delta u + \C u &= f \quad&&\text{in }\Omega \,,\\
	\dfdx{u}{\nu} + \alpha u - \beta \lbg u &=g \quad&&\text{on }\Ga = \partial \Omega \,,
\end{aligned}\right.
\end{align}
where $\Omega \subset \R^n$ ($n=2, 3$) is a domain with curved boundary $\Ga = \partial \Om$, $\alpha > 0, \beta > 0$ and $\C \ge 0$ are given constants and $f$, $g$ are given functions on $\Omega$ and $\partial \Omega$, respectively.

The generalized Robin problem \eqref{eq: GRP} is studied in \cite{KCDQ15} (with $\C = 0$). The authors prove existence and uniqueness of the weak solution and analyze the regularity of the solution given the regularity of $f$ and $g$. It turns out that the solution to the generalized problem possesses better regularity properties than the solution to the standard Robin problem, that is \eqref{eq: GRP} with $\beta = 0$. Moreover, they analyze the \emph{conforming} finite element discretization of \eqref{eq: GRP} and prove optimal-order error bounds in both $L^2$- and $H^1$-norm. However, in \cite{KCDQ15} the authors have to assume that $\Omega$ can be represented exactly by the finite element mesh such that the numerical domain coincides with the exact domain or, equivalently, that the finite element space $V_h$ is contained in the solution space $V$. Two different cases are considered: either $\Ga$ is polyhedral, or of class $C^{1,1}$. In the first case, they have to introduce mixed boundary conditions, because the generalized boundary condition cannot be imposed on the entire boundary (see \cite[Remark 3.1]{KCDQ15}). In the second case, it is restrictive to assume that the computational mesh is capable of  representing the boundary exactly.

The purpose of this paper is to generalize the results of \cite{KCDQ15} to \emph{non-conforming} finite elements, where the additional error that stems from the approximation of the geometry is taken into account. Based on a polyhedral approximation of $\Om$, on which linear finite elements can be used, we construct a piecewise polynomial approximation domain and isoparametric finite elements of arbitrary order. Since the finite element space is no longer contained in the solution space, we cannot compare the finite element solution and the exact solution directly. To overcome this, we lift the finite element solution to the solution space to be able to analyze the error of the method.

The above setting allows us to treat different types of boundary conditions in a unified setting. Here we focus on the generalized Robin problem, and the convergence results for the isoparametric finite element discretization of \eqref{eq: GRP} with the standard Robin boundary condition ($\beta=0$) or Neumann boundary condition ($\alpha=\beta=0$) are obtained as a consequence. We derive error bounds between the exact solution and the lifted finite element solution that are optimal with respect to the regularity of the right-hand side functions $f$ and $g$. Under suitable regularity assumptions, the error satisfies optimal-order error bounds.

\subsection{Applications}
The problem \eqref{eq: GRP} has applications for example in heat conduction processes, see \cite{goldstein2006}, or in the context of Schr{\"o}dinger operators \cite{gesztesy2008}. Generalized Robin boundary conditions appear also in the context of domain decomposition methods \cite{gerardo2010,quarteroni1999} and in the Schwarz waveform relaxation algorithm \cite{gander2007,halpern2009}. A more comprehensive list of applications can be found in \cite{KCDQ15}.

\subsection{Outline of the paper}
In Section 2, we introduce basic notations and derive a variational form of the  generalized Robin problem. In Section~3, the approximation of the geometry is described, followed by the isoparametric finite element method in Section~4. In Section~5, we derive error estimates in both $L^2$- and $H^1$-norm. We begin by stating the main results in Section 5.1, followed by a convergence proof for the $H^1$-estimate that is clearly separated into stability and consistency, and finally the proof of the $L^2$-estimate.  We finish with some numerical experiments in two and three space dimensions in Section~6.

\section{Continuous problem}
\subsection{Preliminaries}
Let $\Omega \subset \R^n$, ($n=2, 3$) be an open, bounded and connected domain with sufficiently smooth boundary $\Ga = \partial \Om$. In the following, we require $\Ga$ at least of class $C^2$. For a more thorough introduction to the following concepts and definitions, we refer to \cite[Section 2]{DE13acta}, where more details about the following concepts can be found, cf. \cite{DE07a,ER13}.

The outer unit normal on $\Ga$ is denoted by $\nu$. The tangential gradient of a function $w$ defined on some open neighborhood of $\Ga$ is given by
\begin{align}
	\nbg w = \nb w - \left( \nb w \cdot \nu \right) \nu
\end{align}
and depends on values of $w$ on $\Ga$ only. The Laplace--Beltrami operator is given by
\begin{align}
	\lbg w = \nbg \cdot \nbg w = \sum_{j=1}^n \left(\nbg\right)_j \left(\nbg\right)_j w \,.
\end{align}
We denote by $d: \R^n \to \R$ the signed distance function
\begin{align}
	d(x) =
	\begin{cases}
		- \mathrm{dist}(x,\Ga) &\text{if }x \in \Om \,,\\
		0 &\text{if }x \in \Ga \,,\\
		\mathrm{dist}(x,\Ga) &\text{otherwise,}
	\end{cases}
\end{align}
where $\mathrm{dist}(x,\Ga) = \inf\{|x-y|: y \in \Ga\}$ denotes the distance of $x$ to $\Ga$. Since $\Ga$ is a $C^2$-manifold, there exists a $\delta > 0$ and a strip
\begin{align}\label{def: uDelt}
	\mathcal{U}_\delta = \{ x \in \R^n: |d(x)|<\delta \}
\end{align}
such that for each $x \in \uDelt$ there exists a unique $p(x) \in \Ga$ such that
\begin{align}\label{eq: closest point}
	x = p(x) + d(x) \nu(p(x)) \,,
\end{align}
see \cite[Section 2]{DE13acta}. $p(x)$ is the closest point to $x$ on $\Ga$.

We let $c>0$ denote a generic constant that assumes different values on different occurrences. We use the standard notation for Sobolev spaces, i.e.  $H^0(\Om) := L^2(\Om) = \{ u: \Om \to \R: \, \int_\Om u^2 \dx < \infty  \}$, $H^{k+1}(\Om) := \{ u \in L^2(\Om): \, \nb u \in H^k(\Om)^n \}$. It is well known that the trace $\tr u$ of a function $u \in H^k(\Om)$ is  in $H^{k-1/2}(\Ga)$ if $\Ga \in C^{k-1,1}$. Due to the Laplace--Beltrami operator in the boundary condition of~\eqref{eq: GRP}, it turns out that we need $\tr u \in H^1(\Ga)$ to derive a weak formulation. Therefore $H^1(\Om)$ is not the suitable weak solution space. Instead, we work with the space
%
\begin{align}
	H^k(\Om;\Ga) := \left\{ u \in H^k(\Om): \, \tr u \in H^k(\Ga) \right\} 
\end{align}
endowed with the norm
\begin{align}\label{eq: norm}
	\| u \|_{H^k(\Om;\Ga)} = \left( \| u \|_{H^k(\Om)}^2 + \| \tr u \|_{H^k(\Ga)}^2 \right)^{1/2} \,.
\end{align}
Recall that for a function $w \in H^k(\Ga)$, the $H^k(\Ga)$-norm is defined using tangential derivatives, i.e.
\begin{align}
	\| w \|_{H^k(\Ga)} = \left( \| w \|_{L^2(\Ga)}^2 + \| \nbg w \|_{H^{k-1}(\Ga)^n}^2 \right)^{1/2} \,.
\end{align}

It is shown in \cite[Lemma 2.5]{KCDQ15} that the space $H^k(\Om;\Ga)$ with the inner product that induces \eqref{eq: norm} is a Hilbert space.

\subsection{Variational form}
To derive the weak formulation, we make use of the integration by parts formula on $\Ga$: for $w \in H^1(\Ga)$, we have (see \cite{DE13acta})
\begin{align}\label{eq: ibp}
	\int_{\Ga} - \Delta_\Ga u w \dS = \int_{\Ga} \nbg u \cdot \nbg w \dS \,.
\end{align}
We multiply \eqref{eq: GRP} with a test function $\varphi$, integrate over $\Omega$ and obtain
\begin{align}
	\int_{\Om} \nb u \cdot \nb \varphi + \C u \varphi \dx - \int_{\Ga} \dfdx{u}{\nu} \tr \varphi \dS = \int_{\Om} f \varphi \dx \,.
\end{align}
Substituting the boundary condition and using \eqref{eq: ibp} with $w = \tr \varphi$, we arrive at
\begin{align}
	\int_{\Om} \nb u \cdot \nb \varphi + \C u \varphi \dx + \alpha \int_{\Ga} (\tr u) (\tr \varphi) \dS + \beta \int_{\Ga} \nbg (\tr u) \cdot \nbg (\tr \varphi) \dS = \int_{\Om} f \varphi \dx + \int_{\Ga} g (\tr \varphi) \dS \,.
\end{align}
We use the following notation for bilinear forms defined on $H^1(\Om;\Ga) \times H^1(\Om;\Ga)$:
\begin{align}\label{def: bilinear forms}
	m^\Om(u,v) &= \int_{\Om} u v \dx \,,\\
	a^\Om(u,v) &= \int_{\Om} \nb u \cdot \nb v \dx \,,\\
	m^\Ga(u,v) &= \int_{\Ga} (\tr u) (\tr v) \dS \,, \\
	a^\Ga(u,v) &= \int_{\Ga} \nbg (\tr u) \cdot \nbg (\tr v) \dS \,, \\
	a(u,v) &= a^\Om(u,v) + \C m^\Om(u,v) + \alpha m^\Ga(u,v) + \beta a^\Ga(u,v) \,.
\end{align}
The right hand side is denoted by
\begin{align}
	\ell(\varphi) &= \int_\Om f \varphi \dx + \int_\Ga g (\tr \varphi) \dS \,.
\end{align}

The variational form thus reads: find $u \in V=H^1(\Om;\Ga)$ such that
\begin{align}\label{eq: variational form}
	a(u,\varphi) = \ell(\varphi)
\end{align}
for all $\varphi \in H^1(\Om;\Ga)$.
The following regularity result is proved in \cite{KCDQ15}.
\begin{proposition}\label{prop: ex uniq reg}
	Let $\alpha, \beta > 0$, $\kappa \ge 0$ and $j \ge 1$. If $\Ga \in C^{j,1}$, $f \in H^{j-1}(\Om)$, $g \in H^{j-1}(\Ga)$, then there exists a unique solution $u \in H^{j+1}(\Om;\Ga)$ that satisfies the a priori bound
	\begin{align}
		\| u \|_{H^{j+1}(\Om;\Ga)} \le c \left( \| f \|_{H^{j-1}(\Om)} + \| g \|_{H^{j-1}(\Ga)} \right) \,.
	\end{align}
\end{proposition}
Let us remark that for the standard Robin boundary value problem, i.e.~\eqref{eq: GRP} with $\beta=0$, we need $g \in H^{j-1/2}(\Ga)$ to have $u \in H^{j+1}(\Om)$, and the trace theorem then yields $\tr u \in H^{j+1/2}(\Ga)$, so the generalized problem requires less regularity in the data to produce a more regular solution, cf.~\cite[Remark 3.5]{KCDQ15}.

\section{Domain approximation}
Before we describe the finite element method, we need to construct an approximation of $\Om$ and $\Ga$. We follow the construction of \cite{ER13}, which is based on \cite{Len86}, \cite{Ber89} and \cite{demlow2009}.
\subsection{Linear approximation}
Let $\Omh^{(1)}$ be a polyhedral approximation of $\Om$ with boundary $\Gah^{(1)} = \partial \Omh^{(1)}$. We construct $\Omh^{(1)}$ such that the faces of $\Gah^{(1)}$ are simplices whose vertices lie on $\Ga$ (triangles in $\R^3$ and straight lines in $\R^2$). We construct a quasi-uniform triangulation $\calTh^{(1)}$ of $\Omh^{(1)}$ consisting of simplices (tetrahedrons on $\R^3$ and triangles in $\R^2$). We set
\begin{align}
	h = \max\{ \mathrm{diam}(T): T \in \calTh^{(1)} \} 
\end{align}
and assume that $h \le h_0$, where $h_0$ is sufficiently small such that $\Gah^{(1)} \subset \uDelt$, where $\uDelt$ is defined in \eqref{def: uDelt}.

\subsection{Exact triangulation}
Before we define the computational domain, we define an \emph{exact triangulation} of $\Om$. We denote by $\hatT$ the unit $n$-simplex. For each $T \in \calTh^{(1)}$, there exists an affine transformation $\Phi_T: \R^n \to \R^n$ that maps $\hatT$ onto $T$, which we write as
\begin{align}
	\Phi_T(\hatx) = B_T \hatx + b_T \,,
\end{align}
where $B_T \in \R^{n \times n}$, $b_T \in \R^n$. $\Phi_T$ is exactly the map used for linear finite elements. 
We now call $T^c$ a curved simplex if there exists a $C^1$-mapping $\Phi_T^c$ that maps $\hatT$ onto $T^c$ which is of the form
\[ \Phi_T^c = \Phi_T + \varrho_T \,,  \]
where $\Phi_T$ is an affine map as defined above and $\varrho_T: \hatT \to \R^N$ is a $C^1$-mapping satisfying
\begin{align}\label{eq: derivative bound rho}
	C_T := \sup_{\hatx \in \hatT} |D \varrho_T(\hatx) B_T^{-1} | \le C < 1 \,.\
\end{align}

There are several ways to define $\varrho_T$. We follow the construction of \cite{ER13}, based on \cite{Dub90}. Note that each $T \in \calTh^{(1)}$ is either an internal simplex with at most one node on the boundary, or $T$ has more than one node on the boundary. In the first case, we set $\varrho_T = 0$. For the latter case, we denote by $l$ the number of nodes of $T$ that lie on the boundary $\Gah^{(1)}$. The vertices $x_1^T,\ldots,x_{n+1}^T$ of $T$ are ordered such that $x_1^T,\ldots,x_l^T$ lie on $\Gah^{(1)}$. For each $x^T \in T$, there is a unique representation
\[ x^T = \sum_{j=1}^{n+1} \lambda_j x_j^T \]
in barycentric coordinates. Note that
\[ \lambda_{n+1} = 1- \sum_{j=1}^n \lambda_j \,.\]
We write $\hatx^T=(\lambda_1,\ldots,\lambda_N)$ for the coordinates of $x$ in $\hatT$. We introduce
\[ \lambda^*(\hatx)=\sum_{j=1}^l \lambda_j \,, \quad \hatsig = \{ \hatx \in \hatT: \lambda^*(\hatx) = 0 \}\,. \]
We have $\lambda^*(\hatx) = 0$ if $\hatx$ is a node which is not belonging to the boundary (or if $\hatx$ is on the edge between such nodes in the three-dimensional case, when $l=2$), and $\lambda^*(\hatx)=1$ if $\hatx \in T \cap \Gah^{(1)}$.

We denote by $\tau_T$ the face of $\Gah^{(1)}$ that corresponds to the boundary face of $T$, i.~e. $\tau_T = T \cap \Gah^{(1)}$. For $\hatx \notin \hatsig$, we denote the projection of $x=\Phi_T(\hatx)$ onto $\tau_T$ by
\[ y(\hatx) = \sum_{j=1}^l \frac{\lambda_j}{\lambda^*} x_j^T \,. \]
Then, using the normal projection $p$ defined in \eqref{eq: closest point}, we define $\varrho_T$ by
\begin{align}
	\varrho_T(\hatx) =
	\begin{cases}
		(\lambda^*(\hatx))^{k+2}(p(y(\hatx))-y(\hatx)) \,,\quad&\text{if } \hatx \notin \hatsig \,,\\
		0\,,\quad&\text{if }\hatx \in \hatsig\,.
	\end{cases}
\end{align}

Basic regularity properties of the above maps are stated and proved in \cite{ER13}. In particular, it is shown that $\rho_T$ satisfies \eqref{eq: derivative bound rho} for $h \le h_0$ sufficiently small.

\subsection{Computational domain and lifts}
We can now define the higher-order computational domain $\Omh^{(k)}$ for $k \ge 1$. Let $T \in \calTh^{(1)}$ and $\varphi_1^k,\ldots,\varphi_{n_k}^k$ be the Lagrangian basis functions of degree $k$ on $\hatT$ corresponding to the nodal points $\hatx^1,\ldots,\hatx^{n_k}$ on $\hatT$. Here, $n_k$ denotes the number of nodal points on each element, for example $n_k=4$ or $n_k=10$ for linear or quadratic finite elements in three  dimensions. Then, we define a parametrization of a polynomial simplex $T^{(k)}$ by
\[ \Phi_T^{(k)}(\hatx) = \sum_{j=1}^{n_k} \Phi_T^c(\hatx^j) \varphi_j^k(\hatx) \,. \]
Note that, by the Lagrangian property, we have
\[ \Phi_T^{(k)}(\hatx^l) = \Phi_T^c(\hatx^l) \,. \]
We can apply this to each $T \in \calTh^{(1)}$ and then define $\Omh^{(k)}$ as the union of elements in $\calTh^{(k)}$, defined by
\[ \calTh^{(k)} := \{ T^{(k)}:T \in \calTh^{(1)} \} \,,\quad T^{(k)} := \{ \Phi_T^{(k)}(\hatx): \hatx \in \hatT \} \,. \]
For $k=1$, this notation is consistent with the notation of $\Omh^{(1)}$ in the previous subsection. 
\begin{definition}
	For a function $w_h: \Omh^{(k)} \to \R$, its lift $w_h^l: \Om \to \R$ is defined by $w_h^l = w_h \circ (\Phi_T^{(k)})^{-1}$, i.e.
	\begin{align}
		w_h^l\left( \Phi_T^{(k)}(x) \right)= w_h(x) \,,\quad x \in \Omh^{(k)} \,.
	\end{align}
	For a continuous function $w: \Om \to \R$, its inverse lift is defined by  $w^{-l} = w \circ \Phi_T^{(k)}$.
\end{definition}
The following lemma states that both the $L^2$-norm and the $H^1$-seminorm of functions on $\Omh^{(k)}$ and their lifts are equivalent.
\begin{proposition}\label{prop: norm equivalence}
	There exists a constant $c > 0$ independent of $h$ (but depending on $k$, $n$ and the geometry of $\Om$), such that for all $w_h: \Omh^{(k)} \to \R$
	\begin{align}
		\frac{1}{c} \lVert w_h \rVert_{L^2(\Omh^{(k)};\Gah^{(k)})} &\le \lVert w_h^l \rVert_{L^2(\Om;\Ga)} \le c \lVert w_h \rVert_{L^2(\Omh^{(k)};\Gah^{(k)})} \,,\\
		\frac{1}{c} \lVert \nb w_h \rVert_{L^2(\Omh^{(k)})} &\le \lVert \nb w_h^l \rVert_{L^2(\Om)} \le c \lVert \nb w_h \rVert_{L^2(\Omh^{(k)})} \,,\\
		\frac{1}{c} \| \nb_{\Gah} (\trh w_h) \|_{L^2(\Gah^{(k)})} &\le \| \nbg (\tr w_h^l) \|_{L^2(\Ga)} \le c \| \nb_{\Gah} (\trh w_h) \|_{L^2(\Gah^{(k)})} \,.
	\end{align}
\end{proposition}
\begin{proof}
	See \cite[Proposition 4.9]{ER13} for the bulk estimate and \cite{demlow2009} for the estimate on the boundary.
\end{proof}

\section{The isoparametric finite element method}
In this section we introduce the finite element method. We use piecewise polynomial finite element functions of degree $k$, which leads to isoparametric finite elements. Isoparametric finite elements are also used in \cite{ER13} in the context of bulk--surface equations; the traces of isoparametric bulk finite element functions on the boundary can be considered as surface finite elements, see e.g. \cite{DE07a,DE13acta}.

From now on, we write $\Omh$ and $\Gah$ instead of $\Omh^{(k)}$ and $\Gah^{(k)}$. We collect the nodes $x_1,\ldots,x_N \in \R^n$ of the triangulation in a vector $\bfx = (x_1,\ldots,x_N) \in \R^{n N}$ such that exactly the first $N_\Ga$ nodes $x_1,\ldots,x_{N_\Ga}$ lie on $\Ga$. We use Lagrangian basis functions $\varphi_1,\ldots,\varphi_N$, which are defined elementwise such that their pullback to the reference element is polynomial of degree $k$. The basis functions satisfy the property $\varphi_j(x_k) = \delta_{jk}$ for $1 \le j,k \le N$. The finite element space is then defined as
\begin{align}
	V_h = \mathrm{span}\left\{ \varphi_1,\ldots,\varphi_N \right\} \,.
\end{align}
Recall that, as opposed to \cite{KCDQ15}, the finite element space $V_h$ is \emph{not} contained in $V = H^1(\Om;\Ga)$. The right-hand side functions are approximated with appropriate functions $f_h: \Omh \to \R$ and $g_h: \Gah \to \R$. If $f$ and $g$ are continuous, one could use the inverse lifts or the finite element interpolations, for example.

We use the following discrete analogues of the bilinear forms defined in \eqref{def: bilinear forms}:
\begin{align}\label{def: discrete bilinear forms}
	m_h^\Om(u_h,v_h) &= \int_{\Omh} u_h v_h \dx \,,\\
	a_h^\Om(u_h,v_h) &= \int_{\Omh} \nb u_h \cdot \nb v_h \dx \,,\\
	m_h^\Ga(u_h,v_h) &= \int_{\Gah} (\trh u_h) (\trh v_h) \dS_h \,, \\
	a_h^\Ga(u_h,v_h) &= \int_{\Gah} \nb_{\Gah} (\trh u_h) \cdot \nb_{\Gah} (\trh v_h) \dS_h \,,\\
	a_h(u_h,v_h) &= a_h^\Om(u_h,v_h) + \C m_h^\Om(u_h,v_h) + \alpha m_h^\Ga(u_h,v_h) + \beta a_h^\Ga(u_h,v_h) \,.
\end{align}
Here, $\trh$ denotes the discrete trace operator on $\Gah$, $\dS_h$ denotes the discrete surface measure on $\Gah$ (see \cite{ER13,demlow2009,DE13acta} for further details).  Moreover, we denote
\begin{align}
	\ell_h(\varphi_h) = \int_{\Omh} f_h \varphi_h \dx + \int_{\Gah} g_h (\trh \varphi_h) \dS_h \,.
\end{align}
The bilinear forms are defined on $V_h \times V_h$ and $\ell_h$ is defined on $V_h$.

The discretized formulation of \eqref{eq: variational form} now reads: given $f_h, g_h \in V_h$, find $u_h \in V_h$ such that
\begin{align}\label{eq: discrete problem}
	a_h(u_h,\varphi_h) = \ell_h(\varphi_h)
\end{align}
for all $\varphi_h \in V_h$. Since $a_h$ is coercive and bounded  and $V_h$ is a (finite-dimensional) Hilbert space, we get existence and uniqueness of the discrete solution by the Lax-Milgram lemma. 

\subsection{Matrix--vector formulation}
We derive a matrix--vector formulation of the discretized problem. First, we note that \eqref{eq: discrete problem} is equivalent to: find $u_h \in V_h$ such that
\begin{align}
	a_h(u_h,\varphi_j) = \ell_h(\varphi_j)
\end{align}
for all basis functions $\varphi_j$, $j=1, \ldots, N$. The functions $f_h$ and $g_h$, which are assumed to be finite element functions, can be written as $f_h(\cdot) = \sum_{j=1}^N f_h(x_j) \varphi_j(\cdot)$, $g_h(\cdot) = \sum_{j=1}^{N_\Ga} g_h(x_j) \varphi_j(\cdot)$. We collect the nodal values in vectors
\begin{align}
	\bff = (f_h(x_j))_{j=1}^N \,,\quad \bfg= (g_h(x_j))_{j=1}^{N_\Ga} \,.
\end{align}
We define the bulk and surface mass and stiffness matrices:
\begin{align}
	(\bfM_\Om)_{jk} &= \int_{\Omh} \varphi_j \varphi_k \dx \,,\\
	(\bfA_\Om)_{jk} &= \int_{\Omh} \nb \varphi_j \cdot \nb \varphi_k \dx \,, \quad 1 \le j,k \le N \,,\\
	(\bfM_\Ga)_{jk} &= \int_{\Gah} (\trh \varphi_j)  (\trh \varphi_k) \dS_h \,,\\
	(\bfA_\Ga)_{jk} &= \int_{\Gah} \nb_{\Gah} (\trh \varphi_j) \cdot \nb_{\Gah} (\trh \varphi_k) \dS_h \,,
	\quad 1 \le j,k \le N_\Ga \,.\\
\end{align}
We introduce the matrix $\bftr = (I_{N_\Ga} , 0) \in \R^{N_\Ga \times N}$, where $I_{N_\Ga}$ denotes the identity matrix of size $N_\Ga \times N_\Ga$. For a finite element function $w_h$ with nodal values collected in a vector $\bfw$, $\bftr \bfw \in \R^{N_\Ga}$ is the vector of the nodal values on the boundary nodes.

\begin{proposition}
	Let $u_h(\cdot) = \sum_{j=1}^N u_j \varphi_j(\cdot) \in V_h$ denote the finite element solution to \eqref{eq: discrete problem} and $\bfu = (u_j)_{j=1}^N$ the vector of nodal values. Then the spatially discretized problem \eqref{eq: discrete problem} is equivalent to the linear system
	\begin{align}\label{eq: matrix vector}
		 \bfK \bfu = \bfb \,,
	\end{align}
	where $\bfK = \bftr\T (\alpha \bfM_\Ga + \beta \bfA_\Ga) \bftr + \C \bfM_\Om + \bfA_\Om$ and $\bfb = \bfM_\Om \bff + \bftr\T \bfM_\Ga \bfg$.
\end{proposition}
\begin{proof}
	Follows from linearity and a direct computation.
\end{proof}
The following properties of $\bfK$ are needed in the error analysis.
\begin{lemma}\label{lemma: K norm}
	For a finite element function $w_h = \sum_{j=1}^N w_j \varphi_j$ with corresponding nodal vector $\bfw \in \R^n$, the $a_h$-norm of $w_h$, defined by $\| w_h \|_{a_h} = (a_h(w_h,w_h))^{1/2} = ( \bfw \T \bfK \bfw)^{1/2}$ and the $H^1(\Omh,\Gah)$-norm are equivalent.
\end{lemma}
\begin{proof}
	For $\alpha=\beta=\kappa=1$, we have $\| w_h \|_{a_h} = \| w_h \|_{H^1(\Omh;\Gah)}$. In the general case, denote $c_1=\min(\alpha,\beta,\kappa,1)$ and $c_2 = \max(\alpha,\beta,\kappa,1)$ and we have
	\begin{align}
		c_1 \| w_h \|_{a_h} \le \| w_h \|_{H^1(\Omh;\Gah)} \le c_2 \| w_h \|_{a_h} \,.
	\end{align}
\end{proof}

\begin{remark}
	If the right-hand side functions $f$ and $g$ are not approximated with finite element functions, the vector $\bfb$ in \eqref{eq: matrix vector} is defined by integrals over $\Omega$ and $\Gamma$, which then have to be approximated with quadrature rules. In this paper, we do not intend analyzing these numerical integration errors and therefore assume that $f$ and $g$ are approximated with finite element functions $f_h$ and $g_h$. This is not fully practical for $f \in L^2(\Om)$, $g \in L^2(\Ga)$, cf. \cite{Dzi88,ER13}. We will carefully carry out the error analysis such that this approximation error is taken into account. If $f$ and $g$ are continuous, $f_h$ and $g_h$ can be chosen as finite element interpolations of $f$ and $g$, and provided that $f$ and $g$ are sufficiently regular this interpolation error is of the same order as the order of the finite element method.
\end{remark}

\begin{definition}
	For a function $w \in H^2(\Om)$, its finite element interpolation $\widetilde{I}_h w \in V_h$ is given by
	\begin{align}
		\widetilde{I}_h w (\cdot) = \sum_{j=1}^N w(x_j) \varphi_j(\cdot) \,.
	\end{align}
	The lifted finite element interpolation $I_h w: \Om \to \R$ is then defined as
	\begin{align}
		I_h w = \left( \widetilde{I}_h w \right)^l \,.
	\end{align}
\end{definition}
Note that since $n \in \{2,3\}$, we have $H^2(\Om) \subset C^0(\Om)$, so the pointwise evaluation is well-defined. The following two approximation properties are crucial in order to prove optimal-order error bounds with respect to the regularity of the exact solution.
\begin{proposition}\label{prop: interpolation}
	Let $k \ge 1$. There exists a constant $c$ independent of $h$ and $j$, such that for all $2 \le j \le k+1$
	\begin{align}\label{eq: approximation property}
		\lVert w- I_h w \rVert_{L^2(\Om;\Ga)} &\le c h^j \lVert w \rVert_{H^j(\Om;\Ga)} \,,\\
		\lVert w - I_h w \rVert_{H^1(\Om;\Ga)} &\le c h^{j-1} \lVert w \rVert_{H^j(\Om;\Ga)} 
	\end{align}
	for all $w \in H^j(\Om;\Ga)$. 
\end{proposition}
\begin{proof}
	See \cite[Corollary 4.1]{Ber89} and \cite{demlow2009}.
\end{proof}

\begin{proposition}\label{prop: bilinear form errors}
	For any $u_h, w_h \in V_h$ with lifts $u_h^l, w_h^l \in V_h^l \subset H^1(\Om)$, we have the following estimates:
	\begin{align}
		\left| m_h^\Om(u_h,w_h) - m^\Om(u_h^l,w_h^l) \right| &\le c h^{k} \| u_h^l \|_{L^2(\Om)} \| w_h^l \|_{L^2(\Om)} \,, \\
		\left| m_h^\Om(u_h,w_h) - m^\Om(u_h^l,w_h^l) \right| &\le c h^{k+1} \| u_h^l \|_{H^1(\Om)} \| w_h^l \|_{H^1(\Om)} \,, \\
		\left| a_h^\Om(u_h,w_h) - a^\Om(u_h^l,w_h^l) \right| &\le c h^{k} \| u_h^l \|_{H^1(\Om)} \| w_h^l \|_{H^1(\Om)} \,.
	\end{align}
	The traces of $u_h, w_h$ on $\Gah$ and their lifts on $\Ga$ satisfy
	\begin{align}
		\left| m_h^\Ga(u_h,w_h) - m^\Ga(u_h^l,w_h^l) \right| &\le c h^{k+1} \| u_h^l \|_{L^2(\Ga)} \| w_h^l \|_{L^2(\Ga)} \,,\\
		\left| a_h^\Ga(u_h,w_h) - a^\Ga(u_h^l,w_h^l) \right| &\le c h^{k+1} \| \nbg u_h^l \|_{L^2(\Ga)} \| \nbg w_h^l \|_{L^2(\Ga)} \,.
	\end{align}
	For $u, w \in H^2(\Om)$ with inverse lifts $u^{-l}, w^{-l}$, we have
	\begin{align}
		\left| a_h^{\Om} (u^{-l},w^{-l}) - a^\Om(u,w) \right| \le c h^{k+1} \| u \|_{H^2(\Om)} \| w \|_{H^2(\Om)} \,.
	\end{align}
\end{proposition}
\begin{proof}
	See \cite[Lemma 7.15]{ER17pre} or in the proof of \cite[Lemma 6.2]{ER13}.
\end{proof}

\section{Error analysis}
In this section, we analyze the error of the isoparametric finite element method. Since the exact solution and the numerical solution are defined on different domains, we cannot compare them directly. Instead, we compare the exact solution to the lift of the numerical solution. We derive optimal-order error estimates for finite elements of arbitrary order $k \ge 1$,  with respect to both the regularity of the solution and the approximation of the data.

We begin by stating the main results of this paper. The proof of the following theorems follows down below and is clearly separated into stability and consistency.

\subsection{Statement of the main result}
\begin{thm}\label{theo: H1 estimate}
	Let $j \ge 1$ be a natural number, $f \in H^{j-1}(\Om)$, $g \in H^{j-1}(\Ga)$, let $u \in H^{j+1}(\Om;\Ga)$ be the solution of \eqref{eq: variational form}. Denote by $u_h : \Omh^{(k)} \to \R$ the numerical solution to \eqref{eq: discrete problem} computed with isoparametric finite elements of order $k \ge 1$, $f_h$ and $g_h$ approximations to $f$ and $g$. Then, the error between the exact solution and the lifted finite element solution is bounded by
	\begin{align}
		\lVert u - u_h^l \rVert_{H^1(\Om;\Ga)} \le C h^{\min(k,j)} + c \lVert f - f_h^l \rVert_{L^2(\Om)} + c \lVert g-g_h^l \rVert_{L^2(\Ga)} \,,
	\end{align}
	where $C$ depends on $\lVert f \rVert_{L^2(\Om)}$, $\lVert g \rVert_{L^2(\Ga)}$ and $\lVert u \rVert_{H^{\min(k,j)+1}(\Om;\Ga)}$.
	
	In particular: If $j \ge k$ and $f_h$ and $g_h$ are chosen such that $\lVert f- f_h^l \rVert_{L^2(\Om)} \le c h^k$ and $\lVert g-g_h^l \rVert_{L^2(\Ga)} \le c h^k$, then the error is bounded by
	\begin{align}\label{eq: optimal order H1}
		\lVert u -u_h^l \rVert_{H^1(\Om;\Ga)} \le C h^k \,,
	\end{align}
	where $C$ depends on the regularity of $f$ and $g$ and on $\lVert u \rVert_{H^{k+1}(\Om;\Ga)}$.
\end{thm}
\begin{remark}
	The assumptions in the second part of Theorem~\ref{theo: H1 estimate} are satisfied if $f \in H^k(\Om)$, $g \in H^{k}(\Ga)$ for $k \ge 2$. In this case $f_h$ and $g_h$ can be chosen as finite element interpolations of $f$ and $g$. The interpolation errors are then bounded using Proposition~\ref{prop: interpolation}, and we arrive at~\eqref{eq: optimal order H1}.
\end{remark}
For the $L^2$-estimate, we need slightly more assumptions, see Remark~\ref{rem: assumption on f}.
\begin{thm}\label{theo: L2 estimate}
	Let $j \ge 1$ be a natural number, $f \in H^{j-1}(\Om) \cap H^1(\Om)$, $g \in H^{j-1}(\Ga)$, let $u \in H^{j+1}(\Om;\Ga)$ be the solution of \eqref{eq: variational form}. Denote by $u_h : \Omh^{(k)} \to \R$ the numerical solution of \eqref{eq: discrete problem} computed with isoparametric finite elements of order $k \ge 1$. Then, the error between the exact solution and the lifted finite element solution is bounded by
	\begin{align}
		\lVert u - u_h^l \rVert_{L^2(\Om;\Ga)} \le C h^{\min(k,j)+1} + c \lVert f - f_h^l \rVert_{L^2(\Om)} + c \lVert g-g_h^l \rVert_{L^2(\Ga)} + ch^{k+1} \lVert f-f_h^l \rVert_{H^1(\Om)} \,,
	\end{align}
	where $C$ depends on $\lVert f \rVert_{H^1(\Om)}$, $\lVert g \rVert_{L^2(\Ga)}$ and $\lVert u \rVert_{H^{\min(k,j)+1}(\Om;\Ga)}$.
	
	In particular: If $j \ge k$ and $f_h$ and $g_h$ are chosen such that $\lVert f- f_h^l \rVert_{L^2(\Om)} \le c h^{k+1}$, $\lVert f-f_h^l \rVert_{H^1(\Om)} \le c$ and $\lVert g-g_h^l \rVert_{L^2(\Ga)} \le c h^{k+1}$, then the error is bounded by
	\begin{align}\label{eq: final estimate}
		\lVert u - u_h^l \rVert_{L^2(\Om;\Ga)} \le C h^{k+1} \,.
	\end{align}
\end{thm}
\begin{remark}
	The assumptions in the second part of Theorem~\ref{theo: L2 estimate} are satisfied if $f \in H^{k+1}(\Om)$, $g \in H^{k+1}(\Ga)$ for $k \ge 1$ with $f_h = \widetilde{I}_h f$, $g_h = \widetilde{I}_h g$, see Proposition~\ref{prop: interpolation}.
\end{remark}
The proof of Theorems~\ref{theo: H1 estimate} and \ref{theo: L2 estimate} follows down below and is clearly separated into stability and consistency.

\subsection{Stability}
The finite element interpolation $u_h^*: \Omh^{(k)} \to \R$ of the exact solution, which corresponds to the nodal vector $\bfu^*=(u(x_j))_{j=1}^N$, satisfies the numerical scheme up to a defect $\bfd$, which corresponds to a finite element function $d_h \in V_h$:
\begin{align}\label{eq: defect equation}
	\bfK \bfu^* = \bfb + (\bfM_\Om + \bftr\T \bfM_\Ga \bftr) \bfd \,.
\end{align}
Note that $\bfK$ is symmetric and positive definite and thus both $\bfK^{-1}$ and $\bfK^{-1/2}$ exist. Subtracting \eqref{eq: defect equation} from \eqref{eq: matrix vector}, we find that the error $\bfe = \bfu - \bfu^*$ satisfies
\begin{align}
	\bfK \bfe = -(\bfM_\Om + \bftr\T \bfM_\Ga \bftr) \bfd \,.
\end{align}
We test this equation with $\bfe$ and obtain
\begin{align}
	\lVert \bfe \rVert_{\bfK}^2 := \bfe\T \bfK \bfe = - \bfe\T (\bfM_\Om + \bftr\T \bfM_\Ga \bftr) \bfd \,.
\end{align}
The defect will be estimated in the dual norm induced by the bilinear form $a_h$:
\begin{align}
	\| \bfd \|_\star &:= \| \bfK^{-1/2} (\bfM_\Om + \bftr\T \bfM_\Ga \bftr) \bfd \|_2 = \sup_{0 \ne \bfw \in \R^N} \frac{\bfd\T (\bfM_\Om + \bftr\T \bfM_\Ga \bftr) \bfK^{-1/2} \bfw}{(\bfw\T \bfw)^{1/2}} \\
	&= \sup_{0 \ne \bfz \in \R^N} \frac{\bfd\T (\bfM_\Om + \bftr\T \bfM_\Ga \bftr) \bfz}{(\bfz\T \bfK \bfz)^{1/2}} = \sup_{0 \ne \varphi_h \in V_h} \frac{\int_{\Omh} d_h \varphi_h \dx + \int_{\Gah} (\trh d_h) (\trh \varphi_h) \dS_h}{\| \varphi_h \|_{a_h}} \,.
\end{align}
With the Cauchy--Schwarz and Young inequality, we obtain
\begin{align}\label{eq: dual norm}
\lVert \bfe \rVert_{\bfK}^2 &= - \bfe\T (\bfM_\Om + \bftr\T \bfM_\Ga \bftr) \bfd = - \bfe\T \bfK^{1/2} \bfK^{-1/2} (\bfM_\Om + \bftr\T \bfM_\Ga \bftr) \bfd \\
&\le \| \bfK^{1/2} \bfe \|_2 \| \bfK^{-1/2} (\bfM_\Om + \bftr\T \bfM_\Ga \bftr) \bfd \|_2 \\
&= \| \bfe \|_{\bfK} \| \bfd \|_{\star} \,.
\end{align}
We thus have shown that
\begin{align}
\lVert \bfe \rVert_{\bfK} \le \rVert \bfd \rVert_\star \,.
\end{align}

\subsection{Consistency}
In this section, we bound the dual norm of the defect in order to obtain an optimal order $H^1$-estimate. In order to prove error bounds of order $j$, we assume that the solution $u \in H^{j+1}(\Om;\Ga)$, which is provided if $\Ga$ is a $C^{j+1}$-manifold, $g \in H^{j-1}(\Ga)$, $f \in H^{j-1}(\Om)$ (see Proposition~\ref{prop: ex uniq reg}). Note that since $j \ge 1$ and the dimension $n \in \{2,3\}$, we have $H^{j+1}(\Om) \subseteq C^0(\Om)$ and the finite element interpolation $\widetilde{I}_h u$ of $u$ is well-defined. 

\begin{proposition}\label{prop: defect estimate}
	Under the assumptions of Theorem~\ref{theo: H1 estimate}, the defect is bounded by
	\begin{align}\label{eq: defect bound}
		\lVert \bfd \rVert_\star \le C h^{\min(k,j)} + c \lVert f_h^l - f \rVert_{L^2(\Om)} + c \lVert g_h^l - g\rVert_{L^2(\Ga)} \,,
	\end{align}
	where $C=C(\lVert f \rVert_{L^2(\Om)},\lVert g \rVert_{L^2(\Ga)},\lVert u \rVert_{H^{j+1}(\Om;\Ga)})$.
\end{proposition}
\begin{proof}
	The defect equation \eqref{eq: defect equation} is equivalent to
	\begin{align}
		m_h^\Om(d_h,\varphi_h) + m_h^\Ga(d_h,\varphi_h) &= \alpha m_h^\Ga(\widetilde{I}_h u ,\varphi_h) + \beta a_h^\Ga(\widetilde{I}_h u, \varphi_h)+  \C m_h^\Om(\widetilde{I}_h u, \varphi_h) + a_h^\Om(\widetilde{I}_h u,\varphi_h) - \ell_h(\varphi_h) 
	\end{align}
	for all finite element functions $\varphi_h \in V_h$. Since $\varphi_h^l \in H^1(\Om;\Ga)$, the exact solution $u$ satisfies
	\begin{align}
		0&= \alpha m^\Ga(u,\varphi_h^l) + \beta a^\Ga(u,\varphi_h^l)+ \C m^\Om(u,\varphi_h^l) + a^\Om(u,\varphi_h^l) - \ell(\varphi_h^l) \,.
	\end{align}
	Subtracting both equations yields
	\begin{align}\label{eq: errors to estimate}
		m_h^\Om(d_h,\varphi_h) +m_h^\Ga(d_h,\varphi_h) &= \alpha \left( m_h^\Ga(\widetilde{I}_h u,\varphi_h) - m^\Ga(u,\varphi_h^l) \right) \\
		&+\beta\left( a_h^\Ga(\widetilde{I}_h u,\varphi_h) - a^\Ga(u,\varphi_h^l) \right) \\
		&+\C \left( m_h^\Om(\widetilde{I}_h u,\varphi_h) - m^\Om(u,\varphi_h^l) \right) \\
		&+\left( a_h^\Om (\widetilde{I}_h u,\varphi_h) - a^\Om(u,\varphi_h^l) \right) \\
		&+\left( \ell_h(\varphi_h) - \ell(\varphi_h^l) \right) \,.
	\end{align}
	We estimate the five terms separately.
	
	(i) We write
	\begin{align}
		&m_h^\Ga(\widetilde{I}_h u, \varphi_h) - m^\Ga(u, \varphi_h^l) \\
		=&m_h^\Ga(\widetilde{I}_h u, \varphi_h) - m^\Ga( I_h u, \varphi_h^l) + m^\Ga(I_h u - u, \varphi_h^l) \,.
	\end{align}
	With the Cauchy--Schwarz inequality and Proposition~\ref{prop: interpolation} we obtain for the second term: 
	\begin{align}
		m^\Ga(I_h u - u,\varphi_h^l) &\le \lVert \tr (I_h u - u) \Vert_{L^2(\Ga)} \lVert \tr \varphi_h^l \rVert_{L^2(\Ga)} \\
		&\le \lVert I_h u - u \rVert_{L^2(\Om;\Ga)} \lVert \varphi_h^l \rVert_{H^1(\Om;\Ga)} \\
		&\le c h^{j+1} \lVert u \rVert_{H^{j+1}(\Om;\Ga)} \lVert \varphi_h^l \rVert_{H^1(\Om;\Ga)} \,.
	\end{align}
	For the first term, we use Proposition~\ref{prop: bilinear form errors} and then Proposition~\ref{prop: interpolation}:
	\begin{align}
		&\left| m_h^\Ga(\widetilde{I}_h u,\varphi_h) - m^\Ga(I_h u, \varphi_h^l) \right| \\
		&\le c h^{k+1} \lVert \tr I_h u \rVert_{L^2(\Ga)} \lVert \tr \varphi_h^l \rVert_{L^2(\Ga)} \\
		&\le c h^{k+1} \lVert I_h u \rVert_{L^2(\Om;\Ga)} \lVert \varphi_h^l \rVert_{H^1(\Om;\Ga)} \\
		&\le c h^{k+1}\left( \lVert I_h u - u \rVert_{L^2(\Om;\Ga)} + \lVert u \rVert_{L^2(\Om;\Ga)} \right) \lVert \varphi_h^l \rVert_{H^1(\Om;\Ga)} \\
		&\le c h^{k+1} \left( ch^{j+1} \lVert u \rVert_{H^{j+1}(\Om;\Ga)} + \lVert u \rVert_{H^1(\Om;\Ga)} \right) \lVert \varphi_h^l \rVert_{H^1(\Om;\Ga)} \\
		&\le c h^{k+1} \lVert u \rVert_{H^{j+1}(\Om;\Ga)} \lVert \varphi_h^l \rVert_{H^1(\Om;\Ga)} \,.
	\end{align}
	
	(ii) Similarly, we write
	\begin{align}
		&a_h^\Ga(\widetilde{I}_h u,\varphi_h) - a^\Ga(u,\varphi_h^l) \\
		= & a_h^\Ga(\widetilde{I}_h u,\varphi_h) - a^\Ga(I_h u , \varphi_h^l) + a^\Ga(I_h u - u,\varphi_h^l) \,.
	\end{align}
	We then proceed as in the first step and obtain
	\begin{align}
		\left| a_h^{\Ga}(\widetilde{I}_h u,\varphi_h) - a^\Ga(u,\varphi_h^l) \right| \le c h^{\min(k+1,j)} \| u \|_{H^{j+1}(\Om;\Ga)} \| \varphi_h^l \|_{H^1(\Om;\Ga)} \,.
	\end{align}

	(iii,iv) Using Propositions~\ref{prop: interpolation} and~\ref{prop: bilinear form errors}, we obtain analogously
	\begin{align}
		\left| m_h^\Om (\widetilde{I}_h u,\varphi_h) - m^\Om(u,\varphi_h^l) \right| \le c h^{\min(k,j)+1} \lVert u \rVert_{H^{j+1}(\Om;\Ga)} \lVert \varphi_h^l \rVert_{H^1(\Om;\Ga)}
	\end{align}
	and
	\begin{align}
		\left| a_h^\Om(\widetilde{I}_h u,\varphi_h) - a^\Om(u,\varphi_h^l) \right| \le c h^{\min(k,j)} \lVert u \rVert_{H^{j+1}(\Om;\Ga)} \lVert \varphi_h^l \rVert_{H^1(\Om;\Ga)} \,.
	\end{align}
	
	(v) For the last term we note that
	\begin{align}
		\ell_h(\varphi_h) - \ell(\varphi_h^l) &= m_h^\Om(f_h,\varphi)-m^\Om(f,\varphi_h^l) \\
		&+m_h^\Ga(g_h,\trh \varphi_h) - m^\Ga(g, \tr \varphi_h^l) \,.
	\end{align}
	We write
	\begin{align}
		m_h^\Om(f_h,\varphi_h) - m^\Om(f,\varphi_h^l) =m_h^\Om(f_h,\varphi_h) - m^\Om(f_h^l,\varphi_h^l) + m^\Om(f_h^l-f,\varphi_h^l)  \,.
	\end{align}
	By the Cauchy--Schwarz inequality the last term is bounded by
	\begin{align}
		m^\Om(f_h^l-f,\varphi_h^l) \le \lVert f_h^l - f \rVert_{L^2(\Om)} \lVert \varphi_h^l \rVert_{H^1(\Om;\Ga)} \,.
	\end{align}
	For the first term we use Proposition~\ref{prop: bilinear form errors}:
	\begin{align}
		\left| m_h^\Om(f_h,\varphi_h) - m^\Om(f_h^l,\varphi_h^l) \right| &\le ch^{k} \lVert f_h^l \rVert_{L^2(\Om)} \lVert \varphi_h^l \rVert_{L^2(\Om)} \\
		&\le ch^{k} \left( \lVert f_h^l - f \rVert_{L^2(\Om)} + \lVert f \rVert_{L^2(\Om)} \right) \lVert \varphi_h^l \rVert_{H^1(\Om;\Ga)} \,,
	\end{align}
	so that we obtain the bound
	\begin{align}
		\left| m_h^\Om(f_h,\varphi_h) - m^\Om(f,\varphi_h^l) \right| \le c h^{k} \lVert f \rVert_{L^2(\Om)} \lVert \varphi_h^l \rVert_{H^1(\Om;\Ga)} + c \lVert f_h^l - f \rVert_{L^2(\Om)} \lVert \varphi_h^l \rVert_{H^1(\Om;\Ga)} \,.
	\end{align}
	In a similar fashion we estimate
	\begin{align}
		\left| m_h^\Ga(g_h,\varphi_h) - m^\Ga(g,\varphi_h^l) \right| \le \left( ch^{k+1} \lVert g \rVert_{L^2(\Ga)} + \lVert g_h^l - g \rVert_{L^2(\Ga)} \right) \lVert \varphi_h^l \rVert_{H^1(\Om;\Ga)} \,.
	\end{align}
	Adding the five estimates, using definition~\eqref{eq: dual norm} of the dual norm together with the coercivity of $a_h$, we obtain \eqref{eq: defect bound}.
\end{proof}
Now we can prove Theorem~\ref{theo: H1 estimate}.

\begin{ProofOf}{Theorem~\ref{theo: H1 estimate}}
	The error is decomposed in the following way:
	\begin{align}
		u-u_h^l = \left( u - I_h u \right) + \left( I_h u - u_h^l \right) \,.
	\end{align}
	With Proposition~\ref{prop: interpolation}, we obtain $\lVert u-I_hu \rVert_{H^1(\Om;\Ga)} \le c h^j \lVert u \rVert_{H^{j+1}(\Om;\Ga)}$. For the second term, we note that $\widetilde{I}_h u - u_h$ is the finite element function corresponding to the nodal vector $\bfu^*-\bfu = -\bfe$, so using Proposition~\ref{prop: norm equivalence} and Lemma~\ref{lemma: K norm}, we obtain
	\begin{align}
		\lVert I_h u - u_h^l \rVert_{H^1(\Om;\Ga)} &\le c \lVert \widetilde{I}_h u -u_h \rVert_{H^1(\Omh;\Gah)} \le c \lVert \bfe \rVert_\bfK \le c \lVert \bfd \rVert_\star \,,
	\end{align}
	so the result follows from Proposition~\ref{prop: defect estimate}.
\end{ProofOf}

\subsection{$L^2$-estimate}
In order to derive an optimal-order $L^2$-estimate, we apply the Aubin--Nitsche trick.

\begin{ProofOf}{Theorem~\ref{theo: L2 estimate}}
	Consider the dual problem: for $\eta \in L^2(\Om;\Ga)$, find $z_\eta \in H^1(\Om;\Ga)$ such that
	\[ a(z_\eta,\psi) = m^\Om(\eta,\psi) + m^\Ga(\tr \eta,\tr \psi) \quad \forall\, \psi \in H^1(\Om;\Gamma)  \,. \]
	This is the weak formulation of \eqref{eq: GRP} with $f=\eta$, $g=\tr \eta$. Since $\eta \in L^2(\Om;\Ga)$, we have $z_\eta \in H^2(\Om;\Ga)$ and $z_\eta$ satisfies the a priori estimate (see Proposition~\ref{prop: ex uniq reg})
	\begin{align}\label{eq: a priori dual problem}
		\lVert z_\eta \rVert_{H^2(\Om;\Ga)} \le c \lVert \eta \rVert_{L^2(\Om;\Ga)} \,.
	\end{align}
	With $\eta = e = u - u_h^l$ and writing $z = z_e$ for brevity, we have
	\begin{align}\label{eq: five terms}
		\lVert e \rVert_{L^2(\Om;\Ga)}^2 &= m^\Om(e,e) + m^\Ga(\tr e, \tr e) = a(e,z) \\
		&= a(u-u_h^l,z-I_h z) + a(u,I_h z) - a(u_h^l,I_h z) \\
		&= a(u-u_h^l,z-I_h z) + \ell(I_h z) - a(u_h^l, I_h z - z) - a(u_h^l,z) \\
		&=a(u-u_h^l,z-I_h z)+\ell(I_h z) - \ell_h(\widetilde{I}_h z) + a_h(u_h,\widetilde{I}_h z) \\
		&-a(u_h^l,I_hz-z) - a(u_h^l-u,z) - a(u,z) \\
		&=a(u-u_h^l,z-I_h z) \\
		&+\left( \ell(I_hz) - \ell_h(\widetilde{I}_h z) \right) \\
		&+\left( a_h(u_h,\widetilde{I}_h z - z^{-l}) - a(u_h^l,I_h z- z) \right) \\
		&+\left( a_h(u_h-u^{-l},z^{-l}) - a(u_h^l-u,z) \right) \\
		&+\left( a_h(u^{-l},z^{-l})- a(u,z) \right) \,.
	\end{align}
	We estimate the five terms separately.
	
	(i) Using the boundedness of $a$, Theorem~\ref{theo: H1 estimate}, Proposition~\ref{prop: interpolation} and the a priori bound \eqref{eq: a priori dual problem}, we obtain
	\begin{align}
		a(u-u_h^l,z-I_h z) &\le c \lVert u - u_h^l \rVert_{H^1(\Om;\Ga)} \lVert z - I_h z \rVert_{H^1(\Om;\Ga)} \\
		&\le c h \lVert z \rVert_{H^2(\Om;\Ga)} \lVert u -u_h^l \rVert_{H^1(\Om;\Ga)} \\
		&\le c h \lVert e \rVert_{L^2(\Om;\Ga)} \left( C h^{\min(k,j)} + c\lVert f-f_h^l \rVert_{L^2(\Om)} + c \lVert g-g_h^l \rVert_{L^2(\Ga)} \right) \\
		&\le \left( C h^{\min(k,j)+1} + c h \lVert f-f_h^l \rVert_{L^2(\Om)} + c h \lVert g-g_h^l \rVert_{L^2(\Ga)} \right) \lVert e \rVert_{L^2(\Om;\Ga)} \,.
	\end{align}
	
	(ii) We write
	\begin{align}
		\ell(I_h z) - \ell_h( \widetilde{I}_h z) &= m^\Om(f,I_h z)-m_h^\Om(f_h,\widetilde{I}_h z) \\
		&+m^\Ga(g, I_h z) - m_h^\Ga(g_h, \widetilde{I}_h z) \\
		&=m^\Om(f-f_h^l,I_hz) + \left( m^\Om(f_h^l,I_h z) - m_h^\Om(f_h, \widetilde{I}_h z) \right) \\
		&+m^\Ga(g-g_h^l, I_h z) + \left( m^\Ga(g_h^l, I_h z) - m_h^\Ga(g_h, \widetilde{I}_h z) \right) \,.
	\end{align}
	Using Cauchy--Schwarz, Proposition~\ref{prop: interpolation} and  \eqref{eq: a priori dual problem}, we see that the first term is bounded by
	\begin{align}
		m^\Om(f-f_h^l,I_hz) &\le \lVert f-f_h^l \rVert_{L^2(\Om)} \left( \lVert I_h z - z \rVert_{L^2(\Om;\Ga)} + \lVert z \rVert_{L^2(\Om;\Ga)} \right) \\
		&\le \lVert f -f_h^l \rVert_{L^2(\Om;\Ga)} \left( ch^2 + 1 \right) \lVert z \rVert_{H^2(\Om;\Ga)} \\
		&\le c \lVert f - f_h^l\rVert_{L^2(\Om)} \lVert e \rVert_{L^2(\Om;\Ga)} \,.
	\end{align}
	For the third term we proceed similarly and obtain 
	\begin{align}
		m^\Ga(g-g_h^l, I_h z) &\le \lVert g-g_h^l \rVert_{L^2(\Ga)} \lVert I_h z \rVert_{L^2(\Ga)} \\
		&\le \lVert g-g_h^l \rVert_{L^2(\Ga)} \left( \lVert I_h z - z\rVert_{L^2(\Om;\Ga)} + \lVert z \rVert_{L^2(\Om;\Ga)} \right) \\
		&\le c \lVert g-g_h^l \rVert_{L^2(\Ga)} \lVert e \rVert_{L^2(\Om;\Ga)} \,.
	\end{align}
	For the second term we use Propositions~\ref{prop: bilinear form errors} and~\ref{prop: interpolation} to obtain
	\begin{align}\label{eq: problem term}
		m^\Om(f_h^l,I_h z) - m_h^\Om(f_h,\widetilde{I}_h z) &\le c h^{k+1} \left( \lVert f_h^l - f \rVert_{H^1(\Om)} + \lVert f \rVert_{H^1(\Om)} \right)  \lVert I_h z \rVert_{H^1(\Om)}  \\
		&\le c h^{k+1} \left( \lVert f_h^l - f \rVert_{H^1(\Om)} + \lVert f \rVert_{H^1(\Om)} \right) \lVert e \rVert_{L^2(\Om;\Ga)} \,.
	\end{align}
	With Proposition~\ref{prop: bilinear form errors}, we similarly obtain
	\begin{align}
		m^\Ga(g_h^l, I_h z) - m_h^\Ga(g_h, \widetilde{I}_h z) \le c h^{k+1} \left( \lVert g \rVert_{L^2(\Ga)} + \lVert g_h^l - g\rVert_{L^2(\Ga)} \right) \lVert e \rVert_{L^2(\Om;\Ga)} \,.
	\end{align}
	
	(iii) With Proposition~\ref{prop: bilinear form errors} and Theorem~\ref{theo: H1 estimate}, we obtain
	\begin{align}
		& a_h(u_h,\widetilde{I_h} z - z^{-l}) - a(u_h^l, I_hz-z) \\ &\le ch^k \lVert u_h^l \rVert_{H^1(\Om;\Ga)} \lVert I_h z - z\rVert_{H^1(\Om;\Ga)} \\
		&\le c h^k \lVert u_h^l - u + u \rVert_{H^1(\Om;\Ga)} c h \lVert z \rVert_{H^2(\Om;\Ga)} \\
		&\le c h^{k+1} \left( \lVert u_h^l - u \rVert_{H^1(\Om;\Ga)} + \lVert u \rVert_{H^1(\Om;\Ga)} \right) \lVert z \rVert_{H^2(\Om;\Ga)} \\
		&\le c h^{k+1} \left( C h^{\min(k,j)} + c \lVert f-f_h^l \rVert_{L^2(\Om)} + c \lVert g-g_h^l \rVert_{L^2(\Ga)} \right) \lVert e \rVert_{L^2(\Om;\Ga)} \,.
	\end{align}
	
	(iv) Using the same arguments, we obtain for the fourth term
	\begin{align}
		& a(u_h-u^{-l},z^{-l}) - a(u_h^l-u,z) \\
		&\le c h^k \lVert u_h^l-u \rVert_{H^1(\Om;\Ga)} \lVert z \rVert_{H^1(\Om;\Ga)} \\
		&\le c h^k \left( C h^{\min(k,j)} + c\lVert f-f_h^l \rVert_{L^2(\Om)} + c \lVert g-g_h^l \rVert_{L^2(\Ga)} \right) \lVert e \rVert_{L^2(\Om;\Ga)} \,.
	\end{align}
	
	(v) For the fifth term, using $u,z \in H^2(\Om)$, we have with Proposition~\ref{prop: bilinear form errors} and Theorem~\ref{theo: H1 estimate}
	\begin{align}
		a_h(u^{-l},z^{-l}) - a(u,z) &\le c h^{k+1} \lVert u \rVert_{H^2(\Om;\Ga)} \lVert z \rVert_{H^2(\Om;\Ga)} \\
		&\le c h ^{k+1} \lVert u \rVert_{H^{k+1}(\Om;\Ga)} \lVert e\rVert_{L^2(\Om;\Ga)} \,.
	\end{align}
	
	Inserting all the bounds into \eqref{eq: five terms} gives the bound:
	\begin{align}
		\lVert e \rVert_{L^2(\Om;\Ga)} \le C h^{\min(k,j)+1}  + c \lVert f - f_h^l\rVert_{L^2(\Om)} + c \lVert g - g_h^l \rVert_{L^2(\Ga)} + c h^{k+1} \lVert f-f_h^l \rVert_{H^1(\Om)} \,,
	\end{align}
	where $C$ depends on $\lVert u \rVert_{H^{j+1}(\Om;\Ga)}$, $\lVert f \rVert_{H^1(\Om)}$ and $\lVert g \rVert_{L^2(\Ga)}$. This completes the proof of Theorem~\ref{theo: L2 estimate}.
\end{ProofOf}
\begin{remark}\label{rem: assumption on f}
	Compared with Theorem~\ref{theo: H1 estimate}, we need for $j=1$ the additional assumption that $f \in H^1(\Om)$. This is due to the first two estimates of Proposition~\ref{prop: bilinear form errors}, which only give a $h^k$-error bound for $f \in L^2(\Om)$ in~\eqref{eq: problem term}. Alternatively, since $f_h^l \in H^1(\Om)$, we could simply estimate
	\begin{align}
		m^\Om(f_h^l,I_h z) - m_h^\Om(f_h,\widetilde{I}_h z) \le c h^{k+1} \| f_h^l \|_{H^1(\Om)} \| e \|_{L^2(\Om;\Ga)}
	\end{align}
	in \eqref{eq: problem term} without using the triangle inequality and then make the reasonable assumption that $f_h$ can be chosen such that $\| f_h^l \|_{H^1(\Om)} \le c \| f \|_{L^2(\Om)}$ with a constant independent of $h$. Keeping in mind that we need $f \in H^{k+1}(\Om)$ anyway to obtain the full order, the assumption $f \in H^1(\Om)$ becomes redundant in this case.
\end{remark}

\begin{corollary}
	Consider the standard Robin problem
	\begin{align}\label{eq: SRP}
		\left\{ \begin{aligned}
		- \Delta u + \C u &= f \quad&&\text{in }\Omega \,,\\
		\dfdx{u}{\nu} + \alpha u &=g \quad&&\text{on }\Ga = \partial \Omega \,,
	\end{aligned}\right.
	\end{align}
	Here, the weak solution $u$ is in $H^1(\Om)$, and with minor modifications to the above convergence proof, we obtain under suitable assumptions the error estimate
	\begin{align}
		\| u - u_h^l \|_{L^2(\Om)} + h \| u - u_h^l \|_{H^1(\Om)} \le C h^{k+1}
	\end{align}
	for the isoparametric finite element method. The same result holds for the Neumann boundary condition, i.e. $\alpha=0$ and $\kappa > 0$.
\end{corollary}

\section{Numerical examples}
We illustrate the theoretical results with some numerical examples. We use isoparametric finite elements of degree one and two to solve a generalized Robin problem in two and three space dimensions. Polyhedral approximations are obtained with \emph{distmesh} \cite{persson2004}. For quadratic finite elements, we add new nodes and project the boundary nodes on the boundary. All functions are implemented in MATLAB, the isoparametric elements are implemented based on the ideas of \cite{BCH06}.

\begin{example}\emph{(Two-dimensional)}\\
	We solve the generalized Robin boundary value problem
	\begin{align}\label{grp}
		\left\{ \begin{aligned}
		- \Delta u + u &=f \quad&&\text{in }\Omega\,,\\
		\dfdx{u}{\nu} + u - \Delta_\Gamma u &=g \quad&&\text{on }\Ga = \partial \Om \,,
		\end{aligned} \right.
	\end{align}
	where $\Omega = \left\{ x \in \R^2:\, |x| < 1 \right\}$ is the unit circle, with isoparametric finite elements of degree one and two. As exact solution, we chose
	\begin{align}
		u(x,y) = xy(x^2+y^2)^2
	\end{align}
	from which we compute the right-hand side functions $f$ and $g$. We compute numerical solutions for different mesh sizes. The finest mesh we used for linear finite elements has around 18000 nodes and the refined version used for quadratic finite elements has around 73500 nodes. The error between the lifted numerical solution and the exact solution is reported in Figure~\ref{fig: GRP order 1} for elements of polynomial degree 1 and 2.
\begin{figure}[htb!]
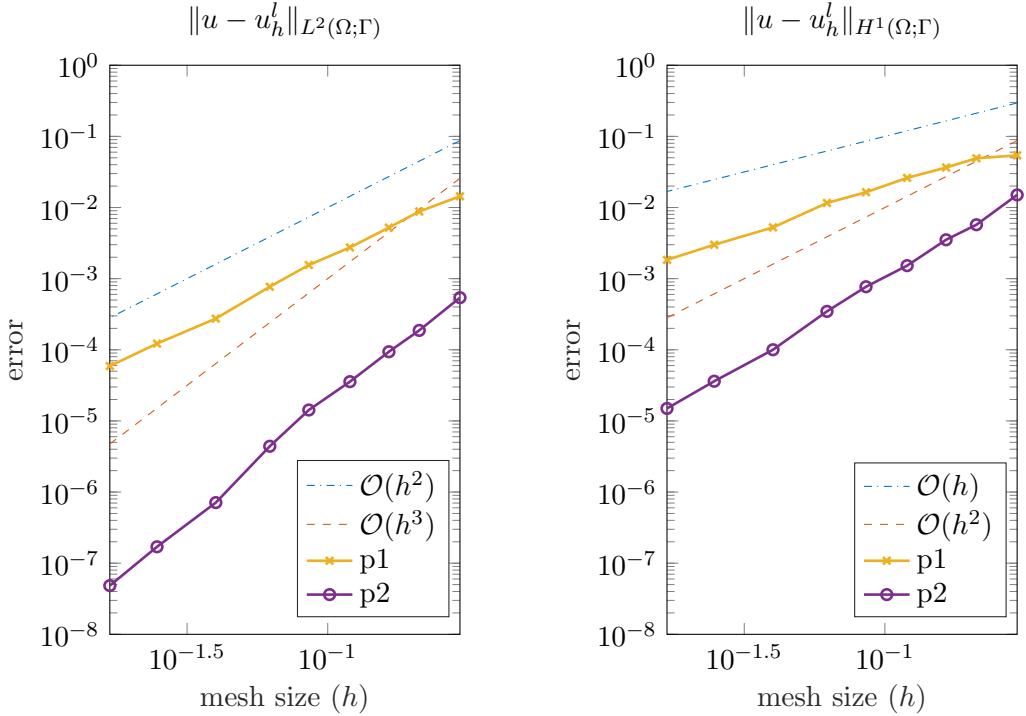

	\begin{center}
		\include{grp_order12}
	\end{center}
	\caption{Convergence rate of the GRP discretization with isoparametric finite elements of degree $1$ and $2$ in two dimensions.}\label{fig: GRP order 1}
\end{figure}
\end{example}

\begin{example}\emph{(Three-dimensional)}\\
	We solve the generalized Robin boundary value problem \eqref{grp} where $\Omega = \left\{ x \in \R^3:\, |x| < 1 \right\}$ is the unit ball, with isoparametric finite elements of degree one and two. As exact solution, we chose
	\begin{align}
		u(x,y) = x^2+y^2-x^2z^2
	\end{align}
	from which we compute the right-hand side functions $f$ and $g$. The finest mesh we used for linear finite elements has around 7000 nodes, and the refined version used for quadratic finite elements has around 55000 nodes. The error between the lifted numerical solution and the exact solution is reported in Figure~\ref{fig: GRP order 1 3d} for elements of polynomial degree 1 and 2.
\begin{figure}[htb!]
	\begin{center}
		\include{grp_order12_3d}
	\end{center}
	\caption{Convergence rate of the GRP discretization with isoparametric finite elements of degree $1$ and $2$ in three dimensions.}\label{fig: GRP order 1 3d}
\end{figure}
\end{example}

%
%

\section*{Acknowledgments}
The author is very grateful to Christian Lubich and Bal\'azs  Kov\'acs for stimulating discussions and their help during the work on this manuscript.

\bibliographystyle{plain}
\bibliography{biblio}
\end{document}

%% file: grp_order12.tex
%
%
\definecolor{mycolor1}{rgb}{0.00000,0.44700,0.74100}%
\definecolor{mycolor2}{rgb}{0.85000,0.32500,0.09800}%
\definecolor{mycolor3}{rgb}{0.92900,0.69400,0.12500}%
\definecolor{mycolor4}{rgb}{0.49400,0.18400,0.55600}%
\begin{tikzpicture}

\begin{axis}[%
width=1.813in,
height=2.971in,
at={(3.089in,0.401in)},
scale only axis,
xmode=log,
xmin=0.0168,
xmax=0.295223302168272,
xminorticks=true,
xlabel style={font=\color{white!15!black}},
xlabel={mesh size ($h$)},
ymode=log,
ymin=1e-08,
ymax=1,
yminorticks=true,
ylabel style={font=\color{white!15!black}},
ylabel={error},
axis background/.style={fill=white},
title style={font=\bfseries},
title={$ \| u - u_h^l \|_{H^1(\Omega;\Gamma)}$},
legend style={at={(0.97,0.03)}, anchor=south east, legend cell align=left, align=left, draw=white!15!black}
]
\addplot [color=mycolor1, dashdotted]
  table[row sep=crcr]{%
0.295223302168272	0.295223302168272\\
0.211533489057441	0.211533489057441\\
0.165026171216324	0.165026171216324\\
0.119948640463718	0.119948640463718\\
0.0856722611339038	0.0856722611339038\\
0.0622338397266234	0.0622338397266234\\
0.0400054046366055	0.0400054046366055\\
0.0247257705206805	0.0247257705206805\\
0.0168036510975332	0.0168036510975332\\
};
\addlegendentry{$\mathcal{O}(h)$}

\addplot [color=mycolor2, dashed]
  table[row sep=crcr]{%
0.295223302168272	0.087156798143139\\
0.211533489057441	0.0447464169928143\\
0.165026171216324	0.0272336371863193\\
0.119948640463718	0.0143876763490942\\
0.0856722611339038	0.00733973632779581\\
0.0622338397266234	0.00387305080711906\\
0.0400054046366055	0.00160043240013853\\
0.0247257705206805	0.000611363727841355\\
0.0168036510975332	0.000282362690207628\\
};
\addlegendentry{$\mathcal{O}(h^2)$}

\addplot [color=mycolor3, line width=1.0pt, mark=x, mark options={solid, mycolor3}]
  table[row sep=crcr]{%
0.295223302168272	0.0540680781194393\\
0.211533489057441	0.0491919496470793\\
0.165026171216324	0.0365016341699942\\
0.119948640463718	0.026063916530555\\
0.0856722611339038	0.0164125827631276\\
0.0622338397266234	0.011575012812463\\
0.0400054046366055	0.00525546592713061\\
0.0247257705206805	0.00299467131469362\\
0.0168036510975332	0.00183548642125359\\
};
\addlegendentry{p1}

\addplot [color=mycolor4, line width=1.0pt, mark=o, mark options={solid, mycolor4}]
  table[row sep=crcr]{%
0.295223302168272	0.015040963389491\\
0.211533489057441	0.00572584521024263\\
0.165026171216324	0.00352014990249127\\
0.119948640463718	0.00152328947307084\\
0.0856722611339038	0.000769410102532125\\
0.0622338397266234	0.000346644747950521\\
0.0400054046366055	0.000100151130494474\\
0.0247257705206805	3.61272819386309e-05\\
0.0168036510975332	1.49793813395523e-05\\
};
\addlegendentry{p2}

\end{axis}

\begin{axis}[%
width=1.813in,
height=2.971in,
at={(0.204in,0.401in)},
scale only axis,
xmode=log,
xmin=0.0168,
xmax=0.295223302168272,
xminorticks=true,
xlabel style={font=\color{white!15!black}},
xlabel={mesh size ($h$)},
ymode=log,
ymin=1e-08,
ymax=1e-00,
yminorticks=true,
ylabel style={font=\color{white!15!black}},
ylabel={error},
axis background/.style={fill=white},
title style={font=\bfseries},
title={$ \| u - u_h^l \|_{L^2(\Omega;\Gamma)}$},
legend style={at={(0.97,0.03)}, anchor=south east, legend cell align=left, align=left, draw=white!15!black}
]
\addplot [color=mycolor1, dashdotted]
  table[row sep=crcr]{%
0.295223302168272	0.087156798143139\\
0.211533489057441	0.0447464169928143\\
0.165026171216324	0.0272336371863193\\
0.119948640463718	0.0143876763490942\\
0.0856722611339038	0.00733973632779581\\
0.0622338397266234	0.00387305080711906\\
0.0400054046366055	0.00160043240013853\\
0.0247257705206805	0.000611363727841355\\
0.0168036510975332	0.000282362690207628\\
};
\addlegendentry{$\mathcal{O}(h^2)$}

\addplot [color=mycolor2, dashed]
  table[row sep=crcr]{%
0.295223302168272	0.0257307177542311\\
0.211533489057441	0.00946536570930916\\
0.165026171216324	0.00449426287315277\\
0.119948640463718	0.00172578221750583\\
0.0856722611339038	0.000628811807328923\\
0.0622338397266234	0.000241034823183317\\
0.0400054046366055	6.40259457610757e-05\\
0.0247257705206805	1.51164392392731e-05\\
0.0168036510975332	4.74472412920983e-06\\
};
\addlegendentry{$\mathcal{O}(h^3)$}

\addplot [color=mycolor3, line width=1.0pt, mark=x, mark options={solid, mycolor3}]
  table[row sep=crcr]{%
0.295223302168272	0.0143850864868182\\
0.211533489057441	0.00877956342679032\\
0.165026171216324	0.00521877946331157\\
0.119948640463718	0.00274592473651216\\
0.0856722611339038	0.00155495918630085\\
0.0622338397266234	0.00076971382490687\\
0.0400054046366055	0.000274542505476727\\
0.0247257705206805	0.000122175633107567\\
0.0168036510975332	5.90517635124108e-05\\
};
\addlegendentry{p1}

\addplot [color=mycolor4, line width=1.0pt, mark=o, mark options={solid, mycolor4}]
  table[row sep=crcr]{%
0.295223302168272	0.000540338492166305\\
0.211533489057441	0.000187299151728961\\
0.165026171216324	9.38764976101426e-05\\
0.119948640463718	3.5640037836927e-05\\
0.0856722611339038	1.42406342037732e-05\\
0.0622338397266234	4.40146504655565e-06\\
0.0400054046366055	7.10634997125758e-07\\
0.0247257705206805	1.70091335674266e-07\\
0.0168036510975332	4.8359730415478e-08\\
};
\addlegendentry{p2}

\end{axis}
\end{tikzpicture}%

%% file: grp_order12_3d.tex
%
%
\definecolor{mycolor1}{rgb}{0.00000,0.44700,0.74100}%
\definecolor{mycolor2}{rgb}{0.85000,0.32500,0.09800}%
\definecolor{mycolor3}{rgb}{0.92900,0.69400,0.12500}%
\definecolor{mycolor4}{rgb}{0.49400,0.18400,0.55600}%
\begin{tikzpicture}

\begin{axis}[%
width=1.813in,
height=2.971in,
at={(3.089in,0.401in)},
scale only axis,
xmode=log,
xmin=0.154112057636748,
xmax=0.764170064213435,
xminorticks=true,
xlabel style={font=\color{white!15!black}},
xlabel={mesh size ($h$)},
ymode=log,
ymin=1e-08,
ymax=1,
yminorticks=true,
ylabel style={font=\color{white!15!black}},
ylabel={error},
axis background/.style={fill=white},
title style={font=\bfseries},
title={$ \| u - u_h^l \|_{H^1(\Omega;\Gamma)}$},
legend style={at={(0.97,0.03)}, anchor=south east, legend cell align=left, align=left, draw=white!15!black}
]
\addplot [color=mycolor1, dashdotted]
  table[row sep=crcr]{%
0.764170064213435	0.764170064213435\\
0.561081903638361	0.561081903638361\\
0.495317573898623	0.495317573898623\\
0.392574194621791	0.392574194621791\\
0.327546821883398	0.327546821883398\\
0.240340872043436	0.240340872043436\\
0.189930235641293	0.189930235641293\\
0.154112057636748	0.154112057636748\\
};
\addlegendentry{$\mathcal{O}(h)$}

\addplot [color=mycolor2, dashed]
  table[row sep=crcr]{%
0.764170064213435	0.0583955887039965\\
0.561081903638361	0.0314812902590447\\
0.495317573898623	0.0245339499012818\\
0.392574194621791	0.0154114498282948\\
0.327546821883398	0.0107286920525914\\
0.240340872043436	0.0057763734774599\\
0.189930235641293	0.0036073494410757\\
0.154112057636748	0.00237505263090323\\
};
\addlegendentry{$\mathcal{O}(h^2)$}

\addplot [color=mycolor3, line width=1.0pt, mark=x, mark options={solid, mycolor3}]
  table[row sep=crcr]{%
0.764170064213435	0.517904712369996\\
0.561081903638361	0.37129789578089\\
0.495317573898623	0.252070717683291\\
0.392574194621791	0.172687504093467\\
0.327546821883398	0.116003207975478\\
0.240340872043436	0.0823026640542884\\
0.189930235641293	0.0591207183952977\\
0.154112057636748	0.0429368767703704\\
};
\addlegendentry{p1}

\addplot [color=mycolor4, line width=1.0pt, mark=o, mark options={solid, mycolor4}]
  table[row sep=crcr]{%
0.764170064213435	0.0519065129031141\\
0.561081903638361	0.0256279788055093\\
0.495317573898623	0.0151241580005022\\
0.392574194621791	0.0097194979570817\\
0.327546821883398	0.00539117595019846\\
0.240340872043436	0.00405725261077272\\
0.189930235641293	0.00206718252050545\\
0.154112057636748	0.00108958039548522\\
};
\addlegendentry{p2}

\end{axis}

\begin{axis}[%
width=1.813in,
height=2.971in,
at={(0.704in,0.401in)},
scale only axis,
xmode=log,
xmin=0.154112057636748,
xmax=0.764170064213435,
xminorticks=true,
xlabel style={font=\color{white!15!black}},
xlabel={mesh size ($h$)},
ymode=log,
ymin=1e-08,
ymax=1,
yminorticks=true,
ylabel style={font=\color{white!15!black}},
ylabel={error},
axis background/.style={fill=white},
title style={font=\bfseries},
title={$ \| u - u_h^l \|_{L^2(\Omega;\Gamma)}$},
legend style={at={(0.97,0.03)}, anchor=south east, legend cell align=left, align=left, draw=white!15!black}
]
\addplot [color=mycolor1, dashdotted]
  table[row sep=crcr]{%
0.764170064213435	0.583955887039965\\
0.561081903638361	0.314812902590447\\
0.495317573898623	0.245339499012818\\
0.392574194621791	0.154114498282948\\
0.327546821883398	0.107286920525914\\
0.240340872043436	0.057763734774599\\
0.189930235641293	0.036073494410757\\
0.154112057636748	0.0237505263090323\\
};
\addlegendentry{$\mathcal{O}(h^2)$}

\addplot [color=mycolor2, dashed]
  table[row sep=crcr]{%
0.764170064213435	0.0446241607697144\\
0.561081903638361	0.0176635822675366\\
0.495317573898623	0.0121520965432533\\
0.392574194621791	0.00605013750429696\\
0.327546821883398	0.00351414898479198\\
0.240340872043436	0.00138829863882129\\
0.189930235641293	0.000685144729383995\\
0.154112057636748	0.000366024247944067\\
};
\addlegendentry{$\mathcal{O}(h^3)$}

\addplot [color=mycolor3, line width=1.0pt, mark=x, mark options={solid, mycolor3}]
  table[row sep=crcr]{%
0.764170064213435	0.180747720494383\\
0.561081903638361	0.162397981018514\\
0.495317573898623	0.111536859778625\\
0.392574194621791	0.0646654213044169\\
0.327546821883398	0.0419292818191011\\
0.240340872043436	0.0247909666387772\\
0.189930235641293	0.0141112730511281\\
0.154112057636748	0.0102976905280704\\
};
\addlegendentry{p1}

\addplot [color=mycolor4, line width=1.0pt, mark=o, mark options={solid, mycolor4}]
  table[row sep=crcr]{%
0.764170064213435	0.00883979126651823\\
0.561081903638361	0.00259045212556109\\
0.495317573898623	0.00134064870430513\\
0.392574194621791	0.000514723567626956\\
0.327546821883398	0.000214984515755581\\
0.240340872043436	8.74902010290355e-05\\
0.189930235641293	3.81333935151846e-05\\
0.154112057636748	1.7592679041029e-05\\
};
\addlegendentry{p2}

\end{axis}
\end{tikzpicture}%